\documentclass[12pt,a4paper]{article}
\usepackage{amsfonts}
\usepackage{amsmath}
\usepackage{amsbsy}
\usepackage{amsxtra}
\usepackage{latexsym}
\usepackage{amssymb}
\usepackage{amsthm}
\usepackage{url}
\usepackage{cite}
\usepackage{pstricks,pst-node}
\usepackage{enumerate}
\makeatletter
\g@addto@macro\thesection
\makeatother

\newtheorem{theorem}{Theorem}
\newtheorem*{theorem*}{Theorem}
\newtheorem*{sht}{Supporting hyperplane theorem}
\newtheorem*{hst}{Hyperplane separation theorem}
\newtheorem*{flm}{The Farkas lemma}

\newtheorem*{shst}{Special hyperplane separation theorem}
\newtheorem{lemma}{Lemma}
\newtheorem{corollary}{Corollary}
\newtheorem{proposition}{Proposition}
\newtheorem{remark}{Remark}
\newtheorem{definition}{Definition}
\newtheorem{example}{Example}
\newtheorem{assertion}{Assertion}

\DeclareMathOperator{\cone}{cone}

\DeclareMathOperator{\bdr}{bdr}
\DeclareMathOperator{\co}{co}

\DeclareMathOperator{\argmin}{argmin}

\DeclareMathOperator{\spa}{span}

\newcommand{\lng}{\langle}
\newcommand{\rng}{\rangle}
\newcommand{\lf}{\left}
\newcommand{\rg}{\right}

\newcommand{\R}{\mathbb R}

\newcommand{\mc}{\mathcal}

\newcommand{\si}{\sigma}
\newcommand{\Si}{\Sigma}

\newcommand{\tp}{^\top}

\newcommand{\cc}{^\circ}

\newcommand{\Hyp}{\mathcal H}

\newcommand{\Conv}{\mathcal C}

\newcommand{\Kup}{\mathcal K}

\newcommand{\Aa}{\mathcal A}

\newcommand{\Be}{\mathcal B}

\newcommand{\f}{\frac}

\begin{document}
\sffamily

\date{}

\title{About the kernel of the strongly quasiconvex function generated projection}
\author{A. B. N\'emeth\\Faculty of Mathematics and Computer Science\\Babes Bolyai University, Str. Kog\u alniceanu nr. 1-3\\RO-400084
Cluj-Napoca, Romania\\email: nemab@math.ubbcluj.ro \and S. Z. N\'emeth\\School of Mathematics, The University of Birmingham\\The Watson Building,
Edgbaston\\Birmingham B15 2TT, United Kingdom\\email: s.nemeth@bham.ac.uk}

\maketitle

\begin{abstract}

This paper explores a natural generalization of Euclidean projection through the lens of strongly quasiconvex functions, as developed in prior works. By establishing a connection between 
strongly quasiconvex functions and the theory of mutually polar mappings on convex cones, we integrate this generalized projection concept into the duality framework of Riesz spaces, vector 
norms, and Euclidean metric projections. A central result of this study is the identification of conditions under which the null space of a projection onto a closed convex cone forms a closed 
convex cone. We provide a comprehensive characterization of such cones and projections, highlighting their fundamental role in extending the duality theory to generalized projection operators.

\end{abstract}

\section{Introduction}

The most natural generalization of Euclidean projection is the concept of projection defined with the help of strongly quasiconvex functions, as interpreted in the works \cite{FN} and 
\cite{FN1}.

The starting point of our present work is the realization that the model relying on strongly 
quasiconvex functions creates a natural connection with the theory of mutually polar mappings
on convex cones. In this way, this concept can be integrated into the duality theory 
encompassing Riesz spaces, vector norms, and Euclidean metric projection.

The generalized projection derived from mutually polar mappings can perfectly fit
into the above-mentioned theory if and only if the null space of the projection
onto the closed convex cone, $P$, i.e., $\ker P =\{x \in \R^n : Px = 0\}$, is also a closed 
convex cone.

In this paper, we characterize those cones and projections for which the null space of the projection is a closed convex cone.

In Section 2, we introduce the terminology used throughout the paper. Section 3 defines the
generalized projection, explores related philosophical considerations, and presents several
examples. Section 4 introduces the concept of mutually polar retractions and provides
corresponding examples. This section also specializes the concept of mutually polar retractions
to generalized projections and addresses the significant challenges associated with projecting 
onto more general cones. Section 6 examines symmetric projections on a minimal wedge. Finally, 
some of the results from Section 6 are extended to general convex cones using the concept of 
the bipolar projector.
\section{Terminology}

Let $\lf(\R^n,\lng\cdot,\cdot\rng\rg)$ (or simply $\R^n$) be the Euclidean
$n$-space equipped with a fixed orthonormal basis $(e^1,\dots,e^n)$, called
standard basis and the corresponding
standard Euclidean inner product $\lng\cdot,\cdot\rng$, given by
\[\lng x,y\rng=x_1y_1+\dots+x_ny_n,\quad\forall x,y\in\R^n,\] where for any 
$z\in\R^n$ the components 
$z_i\in\R$, $i\in\{1,\dots,n\}$ with respect to the standard basis are defined
by \[z=z_1e^1+\dots+z_ne^n.\] Denote the standard Euclidean norm coresponding to
the standard Euclidean inner product by $\|\cdot\|$, that is 
\[\|x\|=\lf( x_1^2+\dots+x_n^2\rg)^{\f12},\quad\forall x\in\R^n.\]

All the topological notion used next will be with respect to
the unique Hausdorff locally convex topology on $\R^n$ defined by this norm. 

For convenience we can identify the elements of $\R^n$ with column vectors 
$(x_1,\dots,x_n)\tp$. 

We need in the proofs some well known 
standard notions (convex set, hyperplane, cone etc.) and some classical results on cones
and convex sets from functional analysis and convex geometry
contained in the monograph \cite{Rockafellar1970} which are well summarized in \cite{Soltan2021}.

The set $\Hyp\subseteq \R^n$ is said \emph{a hyperplane with the normal
$a,\,a\in\R^n, \|a\|=1$} if 
$$\Hyp=\{x\in\R^n:\,\lng a,x \rng=0\}.$$ A hyperplane defines two 
\emph{closed halfspaces} defined by
$$\Hyp_+=\{x\in\R^n:\,\lng a,x \rng \geq 0\},\quad
\Hyp_-=\{x\in\R^n:\,\lng a,x \rng \leq 0\},$$
and two \emph{open halfspaces} defined by
$$\Hyp^+=\{x\in\R^n:\,\lng a,x \rng > 0\},\quad
\Hyp^-=\{x\in\R^n:\,\lng a,x \rng < 0\}.$$

The translate of the hyperplane $\Hyp$ (halfspace $\Hyp_+,\,\,\Hyp^-$, etc.), i.e.,
%The translate of the hyperplane $\Hyp$ (halfspace $\Hyp_+,\,\,\Hyp^_$, etc.), i. e.
the sets $\Hyp +u \,\,(\Hyp_+ +u,\,\Hyp^- +u, $ etc.) are also called hyperplane (halfspace).
Sometimes we also use the notations \textit{}$\Hyp_-(a)=\Hyp_-$ (similarly for the hyperplane
and the other halfspaces) to emphasise that the normal of $\Hyp$ is $a$.

The nonempty  set  $\Conv \subseteq \R^n$ is called \emph{convex} if
$$u,\,v\in \Conv\,\implies \,[u,v]\subseteq \Conv \textrm{, where }
[u,v]=\{tu+(1-t)v :\, t \in [0,1] \subseteq \R\}.$$
If $\Aa \subseteq \R^n$ is a nonempty set,
$$\co \Aa$$
is he smallest  convex set containing $\Aa $ and is called
\emph{the convex hull of $\Aa$.}

The convex set $\Conv$ is called \emph{strictly convex} if from $u,\, v \in \bdr
\Conv$,
$[u,v]\subseteq \bdr \Conv$ it follows that $u=v$, that is, the boundary of $\Conv$ cannot
contain line segments.

\begin{sht} The following statements hold.
	\begin{enumerate}[(i)]
		\item If $\Conv$ is a convex set and $x_0\in \bdr \Conv$, 
			then there exists a hyperplane $\Hyp$
			with the normal $a$ such that 
			$\Conv \subseteq \Hyp_-$, that is,
			$\lng a,x \rng \le \lng a, x_0 \rng,\,
			\forall \,x\in \Conv.$
		\item If $\Conv$ is a strictly convex set and $x_0\in \bdr \Conv$, 
			then there exists a hyperplane $\Hyp$ with the normal $a$ 
			such that $\Conv \subseteq \Hyp_- ,\,\Hyp \cap
			\Conv=\{x_0\}$, that is, $\lng a,x \rng < 
			\lng a, x_0 \rng,\,\forall \,x\in \Conv,\,x\ne x_0.$
	\end{enumerate}
\end{sht}

\begin{hst}
 Let $\Aa$  and $\Be$ be two disjoint, nonempty convex subsets of $\R^n$.
Then, there exists an element $u \in \R^n$ and a hyperplane $\Hyp$ with the normal $a$ such that
$$\Aa \subseteq \Hyp_- +u,\,\, \Be \subseteq \Hyp_+ +u,$$
that is
$$\lng a,x\rng \leq \lng a,u \rng \leq \lng a, y \rng,\,\forall\, x\in \Aa,\,\,\forall \,y \in \Be.$$
\end{hst}
The boundary $\Si$ of a bounded closed convex set with interior points will be called
\emph{closed convex surface.} By the  supporting hyperplane theorem
in each point $x\in \Si$ there exists a supporting hyperplane 
(a \emph{tangent hyperplane})  $\Hyp+x$ with the normal $a$
such that $\Si \subseteq \Hyp_-+x$. The surface is said \emph{smooth} if at each point of $x\in \Si$
there exists a single supporting hyperplane 
with the above property. In this case, the normal $a$ is called \emph{the gradient of $\Si$ at the point $x$}
and is denoted by $\nabla x$ or more specifically $\nabla_{\Si} x$.
The set $\Si$ is called \emph{strictly convex} if it does not contain line segments.

 The following special separation theorem will be frequently used in proofs:

\begin{shst} 
If $\Si$ is a smooth strictly convex surface, $\Conv$ is a nonempty 
convex set with the property that
$\Conv \cap \co \Si= \{x\},$ then the tangent hyperplane $\Hyp$ at $x\in \Si$ of $\Si$
is the separating hyperplane
of $\Conv$ and $\Si$ with
$\Conv \subseteq \Hyp_+ +x$ and $\Si \setminus \{x\}\subseteq \Hyp^-.$
\end{shst}

The set

$$\mc S^n= \{x\in \R^n : \|x\|=1\} $$
is called the \emph{standard Euclidean unit sphere}.

The nonempty set $\Kup \subseteq \R^n$ is called a \emph{cone} if 
\begin{enumerate}[(i)]
	\item $\lambda x\in \Kup$, for all $x\in \Kup$ and $\lambda \in \R_+$ and if 
	\item $\mc S^n \cap \Kup$, is an arcvise connected set. 
\end{enumerate}

$\Kup$ is called \emph{convex cone} if it satisfies

\begin{enumerate}[(i)]
	\item $\lambda x\in \Kup$, for all $x\in \Kup$ , $\lambda \in \R_+$ and if 
   	\item $x+y\in \Kup$, for all $x,y\in \Kup$. 
\end{enumerate}

A convex cone is always a cone, but  for $n\geq 3$ the two notions differ .

A convex cone $\Kup$ is called \emph{pointed} if $\Kup\cap (-\Kup)=\{0\}$. 

A convex cone is called \emph{generating} if $\Kup-\Kup=\R^n$. 

If $\Aa \subseteq \R^m$ is a nonempty set, then
$\cone \Aa$ is the smallest convex cone containing $\Aa$ and it is called 
the \emph{the convex conical hull of $\Aa$}.

The convex cone
$$\Kup= \cone\{x^1,\,...,\,x^n\}$$
with the $n$ linearly independent elements $x^1,\,...,\,x^n \in \R^n$ , 
is called a \emph{simplicial cone}.

The special case
$$\R^n_+:=\cone \{e^1,\,...\,,e^n\}$$
with $e^1,...,e^n$ the standard orthonormal base of $\R^n$ is called the \emph{positive orthant of $\R^n$}. The interior of $\R^n_+$ is denoted
by $\R^n_{++}$.

If $\Kup \subseteq \R^n$ is a convex cone, then the set
$$ \Kup^\circ =\{y\in \R^n : \lng y,x\rng \leq 0,\,\forall x\in \Kup\}$$
is called \emph{the polar of $\Kup$}. The polar is always a closed convex cone.
$\Kup^*= -\Kup^\circ$ is called the \emph{dual } of $\Kup$.

\begin{flm}
If $\Kup \subseteq \R^n$ is a closed convex cone, then $\Kup^{\circ \circ}= \Kup$.
Hence, $\Kup$ and $\Kup^\circ$ are polar to each other.
\end{flm}

\begin{definition}\label{metric}
Consider the function $\si :\R^n \to \R_+$ and the following conditions on it:
\begin{enumerate}
	\item[(i)] Positive definiteness: $\si (x)>0$ if $x\not= 0$ and $\si(0)=0,$
	\item[(ii)] Non-negative homogeneity: $\si(tx)=t\si(x)$ for every $x\in \R^n$
		and every $t\in\R_+,$   
	\item[(iii)] Subadditivity: $\si(x+y)\leq \si(x)+\si(y)$ for every $x,\,y \in \R^n$,
	\item[(iv)] Symmetry:  $\si(-x)=\si(x)$ for every $x\in \R^n$.
\end{enumerate}

The function $\si:\R^n \to \R_+$ satisfying (i)-(iii) is called \emph{an asymmetric norm}. 
\medskip

The function $\si:\R^n \to \R_+$ satisfying (i)-(iv) is called \emph{a norm}.
\medskip

The set
\[\mathbb B =\{x\in \R^n :\si (x) \leq 1\}\] is called \emph{the unit ball}, while 
the set \[\Si =\{x\in \R^n :\si (x) = 1\}\] \emph{the unit sphere} of $\si$.
\end{definition}

Next, we will fix some standard terminology and abbreviations on the geometric 
notions which will be used later. For fixing the terminology, we list some well
known basic results (without citation) which are contained in almost all textbooks
on functional analysis and convex analysis. 

Let $2^{\R^n}$ be the power set of $\R^n$,
that is the family of all subsets of $\R^n$.

\begin{definition}\label{geot1}
	$\,$

\begin{enumerate}
 \item [(i)] $\Aa$ and $\Be$ in $2^{\R^n}$ are called \emph{translation equivalent}
(denote $\Be \equiv \Aa$ ) if there exists an element $v\in \R^n$ such that
$\Be = \Aa +v$ (and hence  $\Aa = \Be- v$).
The relation $\equiv$ is obviously an equivalence relation on $2^{\R^n}$.

\item [(ii)] $\Aa$ and $\Be$ in $2^{\R^n}$ are called \emph{positive homothetic equivalent}
(denote $\Be \sim \Aa$)
if there exists a point in $u\in \R^n$ and $t>0$ such that $t(\Aa-u)=\Be-u$
(and hence $(1/t)(\Be-u) = \Aa-u$).

The relation $\sim$ is also an equivalence relation on  $2^{\R^n}$.
Here $u$ is called \emph{the center of the homothety}.
Considering the translation as a homothety with center at the infinity, 
by unifying the above two relations into one denoted by $\simeq$ and calling
the unified relation {\em similarity}, we conclude that $\simeq$ is a
reflexive transitive and symmetric relation.
\end{enumerate}
\end{definition}

\begin{assertion}
	If $\Aa \sim \Be$, then the similarity relation is one to one
correspondence of points of $\Aa$ and points of $\Be$ 
(which preserves the "shape" of the sets).
\end{assertion}

\begin{assertion}\label{utolag}

The similar image of hyperplanes, halfspaces are hyperplanes, halfspaces.

If $\Si_i,\, i = 1, 2$ are two similar images of $\Si\subseteq \R^n$  which 
is a closed smooth surface and we have the homothety
$$ t(\Si_1-u)=\Si_2-u,$$ $x\in \Si_1$ and $y = u+t(x-u)\in \Si_2$, then 
$\nabla_{\Si_1}x$ is equal to $\nabla_{\Si_2} y.$
The assertion also holds in the case of a translation.
\end{assertion}

%%%%%%%%%%%%%%%%%%%%%%%%%%%%%%%%%%%%%%%%%%%%%%%%%%%%%%%%%%%%%%%%%%%%%%%%

\section{The generalized projection: definition, philosophy, examples}

 A continuous function   $\varphi:\R^n \to \R$  is called \emph{ strongly quasiconvex}  if for each $\lambda \in (0,1)$
and each $x, y \in \R^n $ with  $x\neq y$ we have 
$$
\varphi ( \lambda x+(1-\lambda)y) < \max \{ \varphi (x), \varphi (y)\}.
$$
 
It follows from the definition that the function $\varphi$ is  strongly quasiconvex if and only if all of its  
nonempty sublevel sets $\{x\in \R^n : \varphi (x)\leq L \}$ for $L \in \R$,  are  
strictly convex (remember that $ \varphi $ is continuous). Moreover, for an arbitrary
nonempty convex set $\Conv$ ,
for each $x\in \R^n$ the  function  $\varphi_x(y) = \varphi(x-y)$ has only one minimizer in $\Conv$,
 see Theorem~\(3.5.9\) of \cite{Baz79}.

\begin{definition}\label{foodef1}

Suppose that $\varphi$ is a strongly quasiconvex, continuous,
nonnegative function with $\varphi(0)=0.$
Suppose $\Conv \subseteq \R^n$ is a closed and convex cone 
Then, for any  $x\in \R^n$ there is a well defined
mapping $P:\R^n\to \Conv$ called \emph{the $\varphi$-projection of $x$ onto $\Conv$}
given by the formula:
\begin{equation}\label{fookepl}
Px=\argmin \{\varphi(x-y): y\in \Conv\}.
\end{equation}
Hence, $P$ depends on $\varphi$ and $\Conv $ and if necessary 
we will also use the notation $P_{\Conv}^{\varphi}.$

The \emph{distance of $x$ from $\Conv$} is defined by
\begin{equation}\label{foodef2} 
d(x,\Conv)= \varphi(x-Px).
\end{equation}

\end{definition}

\begin{remark}\label{remm1}
It is immediate from the formula (\ref{fookepl}) that the same generalized projection
is defined by many functions. For instance if $f: \R_+ \to \R_+$ is a continuous,
increasing, strongly quasiconvex function with $f(0)=0$,
then $P_\Conv^{\varphi} = P_\Conv^{f\circ \varphi}$.
\end{remark}

\begin{remark}\label{philosoph}

The above defined generalized projection is a natural 
generalization of the Euclidean metric projection
since this is precisely the projection
with respect to \[\si(x) :=\|x\|= \lf(x_1^2+...+x_n^2\rg)^{1/2}.\] 

\end{remark}

%%%%%%%%%%%%%%%%%%%%%%%%%%%%%%%%%%%%%%%%%%%%%%%%%%%%%%%%%%%%%%%%%%%%%%%%%%%%%%%%%%%%%%%%%%%%%%%%%%%%%%%%%%%%%%%

\subsection{The initial result}

\begin{theorem}\label{footet}
If $\varphi$ is a strongly  quasiconvex, continuous, nonnegative function with $\varphi(0)=0$
and if $\Conv$ is a closed convex cone, then for the $\varphi$-projection $P$ defined by  
(\ref{fookepl}) the following equalities hold:
\begin{equation}\label{foo1}
	P(I-P)=0,\quad (I-P)(\R^n)=\ker(P):=\{x\in\R^n:P(x)=0\}.
\end{equation}
\end{theorem}

\begin{proof}

	Since $\Conv$ is closed under the addition, $Px+v\in \Conv$ for
	any $v\in \Conv$ and since $$Px=\argmin\{\varphi (x-u):u\in \Conv\},$$
we have $$ \varphi(x-Px)\leq \varphi(x-Px-v)$$
and since the minimum point of $\varphi$ is unique, it  follows that
$$\argmin \{\varphi(x-Px-v): v\in \Conv\}=0.$$

Hence,
$$ P(x-Px) =\argmin\{\varphi((x-Px)-u): u\in \Conv\} = \argmin\{\varphi(x-Px-u):
u\in \Conv\}=0,$$
according to the above consideration. The second equality of \eqref{foo1} easily
follows from its first equality.

Since $x \in \R^n$ was arbitrarily chosen it follows that
$$P(I-P)=0.$$

\end{proof}

\begin{remark}\label{foorem}
We state this almost obvious consequence of (\ref{fookepl}) as a theorem,
since it makes connection with another important concept of functional 
analysis, vector lattice theory and optimization theory, called mutually
polar maps, and it is the starting point in our next investigations. 
\end{remark}

\begin{proposition}\label{homo1}

If $\Conv$ is a convex cone and $\varphi $ is a  strongly quasiconvex function
which is non-negatively 
homogeneous, that is, $\varphi(tx)=t\varphi(x)$ for arbitrary $x\in \R^n$ and arbitrary
$t\in \R_+,$ then P is non-negatively homogeneous too,
that is $P(tx) =t Px$ for any $x\in \R^n $ and any $t\in \R_+.$

\end{proposition}

\begin{proof}

Indeed,
\begin{eqnarray*}
	P(tx) = \argmin \{\varphi(tx-u): u\in \Conv\} =\argmin \{t\varphi(x-u/t):
	u/t\in \Conv\}=\\
	t\argmin \{\varphi(x-u/t): u/t\in \Conv \}=tPx.
\end{eqnarray*}

\end{proof}

\begin{remark}\label{remm2}
Obviously, the same proof holds if $\varphi$ is $\alpha$-homogeneous with $\alpha>0$ that is if
$$\varphi(tx) =t^\alpha \varphi(x),\quad\forall t\in \R_+.$$
%(Compare the affirmation with Remark \ref{remm1}.)
\end{remark}

%\begin{remark}\label{ekvasym}
%	Proposition \ref{homo1} makes a connection with the projection with
%	respect an asymmetric norm $\varphi$, ??? since 
%$$ \mbb B=\{x\in \R^n: \varphi(x)\leq 1\}$$
%is the unit ball of an asymmetric  norm ???.
%
%If $\varphi$ is also symmetric then $P$ is projection with respect to a norm.
%
%??? Thus, in these cases the analytic approach of the generalized projection by strongly
%quasiconvex functions has a geometric approach too in the
%context of asymmetric norms and norms ???.
%\end{remark}
%
%%%%%%%%%%%%%%%%%%%%%%%%%%%%%%%%%%%%%%%%%%%%%%%%%%%%%%%%%%%%%%%%%%%%%%%%%%%%%%%%%%%%%%%%%%%%%%%%%%%%%%

\section{Mutually polar retractions: definition \\  and examples}\label{mutu}

The continuous mapping $Q:\R^n \to \R^n$ is called \emph{retraction} if $Q0=0$ and it
is \emph{idempotent}, that is, if $Q^2=Q$.
Let $0$ be the zero mapping and $I$ the identity mapping of $\R^n$. They 
are retractions. 

\begin{definition}\label{muture} 
The mappings $P,R:\,\R^n\to \R^n$ are called
\emph{mutually polar retractions} if 
\begin{enumerate}[(i)]
	\item $P$ and $R$ are retractions, 
	\item $P+R=I$,
\end{enumerate}
\end{definition}

 We will use the notations $P\R^n=\mc P$ and $R\R^n=\mc R$.

 We will say that $P$ and $R$ 
 are \emph{polars of each other}.
Also, the sets $\mc P$ and $\mc R$ are said mutually polar.

The mappings $I$ and $0$ are mutually polar retractions.

\begin{lemma}\label{kov}

If $P$ and $R$ are mutually polar retractions, then
\begin{enumerate}[(i)]
	\item $PR=RP=0.$
	\item  $\mc P \cap \mc R = \{0\}$ and $\mc P + \mc R= \R^n$,
	\item $Px-Rx= 0$ if and only if $x=0$.
\end{enumerate}
    \end{lemma}

\begin{proof}
	$\,$

	\begin{enumerate}[(i)]
		\item From Definition \ref{muture}, it follows that
			\(PR=(P+R)R-R^2=R-R^2=0.\) Hence, by swapping $P$ and
			$R$, we also have $RP=0$.
		\item Let any $z\in\mc P\cap\mc R$. Hence, there exists
			$x,y\in\R^n$ such that $z=Px=Ry$, which, by using the 
			previous item, implies $z=Px=P^2x=PRy=0$. The second
			equality is trivial.
		\item The equality $P0-R0=0$ is trivial. Suppose that $Px-Rx=0$. 
			Hence, by using the first equality of item (ii), we
			obtain $Px=Rx\in\mc P\cap\mc R=\{0\}$. Hence, the second 
			equality of item (ii) yields $x=Px+Rx=0$. 
	\end{enumerate}
\end{proof}

%\end{remark}

\begin{example}\label{ex0}
If $\R^n_+$ is the positive orthant in  $\R^n$, then putting 
$x=t_1e^1+\dots+t_ne^n$ consider the mappings
$$Qx=|t_1|e^1+\dots+|t_n|e^n$$
and
$$Px=t_1^+e^1+\dots+t_n^+e^n, $$
where $t^+$ denotes the positive part of the real number $t$.

$Q$ and $P$ are both retractions on the positive orthant $\R^n_+$. 
$I-P$ is retraction too, but $I-Q$ is not.
\end{example}

\begin{example}\label{eex1}

The \emph{relation $\leq$ induced by the pointed convex cone $\mc K$} is given by $x\leq y$ if and only if 
$y-x\in \mc K$. Particularly, we have $\mc K=\{x\in \R^n:0\leq x\}$. The relation $\leq$ is an \emph{order relation}, that 
is, it is reflexive, transitive and antisymmetric; it is \emph{translation invariant}, that is, $x+z\leq y+z$, 
$\forall x,y,z\in \R^n$ with $x\leq y$; and it is \emph{scale invariant}, that is, $\lambda x\leq\lambda y$, 
$\forall x,y\in \R^n$ with $x\leq y$ and $\lambda \in \R_+$.

Conversely, for every $\leq$ scale invariant, translation invariant and antisymmetric order relation in $\R^n$ there is a pointed 
convex cone $\mc K$, defined by $\mc K=\{x\in \R^n:0\leq x\}$, such that $x\leq y$ if and only if $y-x\in \mc K$. The cone $\mc K$ 
is called the \emph{positive cone of $\R^n$} and $(\R^n,\leq)$ (or $(\R^n,\mc K)$) is called an \emph{ordered vector space}; we also use
the notation $\leq =\leq_\mc K.$ 

The ordered vector space $(\R^n,\leq)$ is called \emph{lattice ordered} if for every  $x,y\in \R^n$ there exists 
$x\vee y:=\sup\{x,y\}$. In this case the positive cone $\mc K$ is called a \emph{lattice cone}.
A lattice ordered vector space is called a \emph{Riesz space}. 
Denote $x^+=0\vee x$ and $x^-=0\vee (-x)$. Then, $x=x^+-x^-$, $x^+$ is 
called the \emph{positive part} of $x$ and $x^-$ is called the \emph{negative part} of $x$. 
The \emph{absolute value} of $x$ is defined by $|x|=x^++x^-$. The mapping 
$x\mapsto x^+$ is called the \emph{positive part mapping}.

The continuity of the positive
part mapping in the vector lattice is equivalent to the closedness of $\mc K$.

If $(\R^n,\mc K)$ is a Riesz space with positive cone $\mc K$ then $P$ and $R$ defined by
$Px=x^+$ and $Rx=-x^-$ are mutually polar retractions. Here $\mc P=\mc K$ andf $\mc R = -\mc K$.

The order theoretic conditions on $\mc K$ impose strong geometric restrictions on it.
This follows from the following result:

\begin{theorem}[Youdine]\label{Yu}
The pair $(\R^m,\mc K)$ is a vector lattice with continuous lattice operations 
if and only if $\mc K$ is a simplicial cone. 
\end{theorem}

\end{example}

\begin{example}\label{eex2}

If $\mc C\subseteq \R^n$ is a nonempty closed convex set, then it is well defined the 
\emph{metric projection} onto  it, i.e., the mapping $P_\mc C: \R^n \to \mc C$ defined by
$$ P_\mc Cx=\argmin \{y-x:\,y\in \mc C\}.$$

If $\mc K\subseteq \R^n$ is a closed convex cone and $\mc K^\circ$ is the polar of  $\mc K$,
then  the metric projections $P_\mc K$ and $P_{\mc K^\circ}$ are mutually polar retractions.
This is the consequence of the following theorem of Moreau:

\begin{theorem} [Moreau]
	Let  $\mc K,\,\,\mc L\subset \R^n$ be  two mutually polar closed 
	convex cones. Then, the following statements are equivalent:
	\begin{enumerate}
		\item[(i)] $z=x+y,~x\in \mc K,\,\,~y\in \mc L$ and $\lng x,y\rng=0$,
		\item[(ii)] $x=P_\mc K(z)$ and $y=P_\mc L(z)$.
	\end{enumerate}
\end{theorem}

\cite{Moreau1962}, \cite{Zarantonello1971}.
\end{example}

Using a geometric method, a large class of mutually polar retractions is defined in \cite{Nemeth2020}.

\subsection{Mutually polar retractions for weighted power-cones}

The following lemma is Example 2.25 in \cite{BV04}, where it is both stated and proved. 

\begin{lemma}\label{dualn}

Let $\|\cdot\| $ be  a norm on $\R^n$. Its \emph{dual norm}  $\|\cdot\|_*$ is defined by
$$ \|y\|_* = \sup \{|\lng y, x\rng | :\,\|x\|\leq 1\}.$$

Consider the following pointed closed convex cones in $\R^{n+1}$:
$$\Kup = \lf\{(x_0,x_1,...,x_n)\tp=\lf(x_0,x\tp\rg)\tp:\,\,x_0\geq \|x\|,\,\,
x=(x_1,...,x_n)\tp \in \R^n\rg\}$$
and
$$\mc L=\lf\{(y_0,y_1,...,y_n)
=\lf(y_0,y\tp\rg)\tp:\,\,y_0\geq \|y\|_*,\,\,y=(y_1,...,y_n)\tp \in \R^n\rg\}.$$

Then, $\mc L = \Kup^*$ and hence $-\mc L$ and $\Kup $ 
are mutually polar pointed closed convex cones, with the metric projections
$P_\Kup$ and $P_{-\mc L}$ mutually polar retractions. 

\end{lemma}

It is easy to verify that $\mc K,\mc L$ in Lemma \ref{dualn} are pointed closed convex cones.
%
%\begin{proof}
%According Moreau theorem and Farkas lemma it is enough to see that
%$\Kup^\circ =-\mc L$ or equivalently , that $\mc L = \Kup^*$ with $\Kup^*$ the dual of $\Kup$.
%
%Using the same symbol for tha scalar product in $\R^n$ and $\R^{n+1}$ we have
%$$ \lng (x_0,x),(y_0,y)\rng =x_0y_0+ \lng x,y \rng,\,\,\textrm{with}\,\, x=(x_1,...,x_n),\,\,y=(y_1,...,y_n).$$
%
%If $(x_0,x) \in \Kup$ and $(y_0,y) \in \mc L$,then
%
%$$x_0y_0 \geq \|x\| \,\|y\|_*= \|x\| \,\|-y\|_* \geq -\lng x,y \rng.$$
%That is
%$$ \lng (x_0,x),(y_0,y)\rng  = x_0y_0 +\lng x, y \rng \geq 0$$
%or by fixing for moment $(y_0,y)$
%$$ \lng (x_0,x),(y_0,y) \rng \geq 0,\,\, \forall \,\, (x_0,x) \in \Kup.$$
%
%Thus, $\mc L \subset  \Kup^* $ .
%
%To see that $\Kup^* \subset \mc L$ it is sufficient to prove that $(1,y) \in \Kup^*$ 
%implies $(1,y) \in \mc L$.
%
%By  definition
%$$\lng (1,y),(1,x) \rng \geq 0,\,\,\forall \,\, (1,x) \in \Kup.$$
%Hence
%$$1+\lng y,x \rng \geq 0, \,\,\forall \,\, \|x\| \leq 1,$$
%whereby by the symmetry of the norm
%$$|\lng y,x \rng |\leq 1 , \,\, \forall \,\, \|x\| \leq 1, $$
%relation which shows that $\|y \|_* \leq 1,$ and thus $(1,y) \in \mc  L.$
%
%\end{proof}

\begin{theorem}\label{ex4}
	Let $n$ be an arbitrary positive integer, $k\ge 1$
	and $p>2-1/k$ be a 
	rational number with even numerator and odd denominator, $\xi\in\R^m_{++}$,
	$\omega\in\R^m_{++}$, $\mc L(\xi,k)$ be the pointed closed convex
	weighted power-cone defined by
	\[\mc L(\xi,k):=\lf\{(x_1,\dots,x_m,x_{m+1})\tp\in\R^{m+1}:
		x_{m+1}\ge\lf(\xi_1|x_1|^k+\dots+\xi_m|x_m|^k\rg)^{\f1k}\rg\},
	\]
	$\|\cdot\|_{p,\omega}$ be the norm in $\R^n$ defined by 
	\[
		\|x\|_{p,\omega}=\lf(\omega_1|x_1|^p+\dots+\omega_s|x_n|^p\rg)^{\f1p}
	\]
	and $P^{\|\cdot\|_{p,\omega}}_{\mc L(\xi,k)}$ be the projection onto $\mc L(\xi,k)$
	with respect to the norm $\|\cdot\|_{p,\omega}$ in $\R^{m+1}$. Denote 
	$P=P^{\|\cdot\|_{p,\omega}}_{\mc L(\xi,k)}$ and $R=I-P$. Then, $P,R$ is a pair of 
	mutually polar retractions.  
\end{theorem}

\begin{proof}
	Denote $\mc P=\mc L(\xi,k)$ and $\varphi=\|\cdot\|_{p,\omega}$. 
	We have $P^2=P$ and $\mc P$ is a pointed
	closed convex cone. Since $P+R=I$ and $PR=0$ (which, by Theorem
	\ref{footet}, holds for any generalized projection), it follows that
	$R^2=(P+R)R=R$. Since the second inequality in \eqref{foo1} implies that 
	$\mc R=\ker(P)$, it is enough to show that $\ker(P)$ is a
	pointed convex cone. 
	%Let any $\lambda>0$. By the definition of the projection, it follows that 
	%$P(\lambda x)=\lambda Qx$ for any $x\in\R^{m+1}$. Hence, $\lambda\ker
	%P\subseteq\ker(P)$. It remains to show that
	%$\ker(P)+\ker(P)\subseteq\ker(P)$.
	Theorem 5 of \cite{FN} yields 
	\[\ker(P)=\lf\{x\in\R^{m+1}\setminus\{0\}:\nabla\varphi(x)\in\mc
	P\cc\rg\}\cup\{0\}.\] 
	For all $i\in\{1,\dots,m\}$ denote $\zeta_i=\xi_i^{-1}$ and
	$\zeta=(\zeta_1,\dots,\zeta_m)\tp$.
	
	It is known that the $\|\cdot\|_{k,\xi}$ and $\|\cdot\|_{\ell,\zeta}$ norms are dual 
	(see \cite{SchaeferWolff1999}), hence via Lemma \ref{dualn} we have
	\begin{equation}\label{cc}
		\mc P\cc=-\mc L(\zeta,\ell),
	\end{equation}
	where $1/k+1/\ell=1$. 
	
	On the other hand, for any non-zero $x\in\R^{m+1}$ we have
	$\nabla\varphi(x)\in\mc P\cc$ if and only if 
	\begin{equation}\label{(p-1)l}
		(-x_{m+1})^{p-1}=-x_{m+1}^{p-1}\ge\lf(\zeta_1\omega_1\lf|x_1^{p-1}\rg|^\ell
			+\dots+\zeta_m\omega_m\lf|x_m^{p-1}\rg|^\ell\rg)^{\f1\ell}. 
		\end{equation} 
		Let $r:=(p-1)\ell$ and
		$\beta=\lf(\zeta_1\omega_1,\dots,\zeta_m\omega_m\rg)\tp$. Then, 
		$r>(1-1/k)\ell=1$ and formula \eqref{(p-1)l}
		implies that $\nabla\varphi(x)\in\mc P\cc$ if and only if 
		$x\in-\mc L(\beta,r)$. Hence, $\mc R=\ker(P)=-\mc L(\beta,r)$, which is a
		pointed closed convex cone. Indeed, Lemma 
		\ref{dualn} implies that $\mc L(\gamma,t)$ is
		a pointed closed convex cone for all $\gamma\in\R^m_{++}$ and
		$t\ge 1$.
\end{proof}

\begin{remark}\label{kulomb}

We observe an essential difference among the above examples.
In the Examples \ref{ex0}, \ref{eex1} and Theorem \ref{ex4} the retractions refer to special
cones, in the case of the metric projection given in Example \ref{eex2}
the retraction comprise the whole class of closed convex cones.

\end{remark}

%%%%%%%%%%%%%%%%%%%%%%%%%%%%%%%%%%%%%%%%%%%%%%%%%%%%%%%%%%%%%%%%%%%%%%%%%%%%%%%%%%%%%%%%%%%%%%

\subsection{Generalized projections engendering mutually polar retractions}

From Theorem \ref{footet} it follows that for a general projection $P$
onto  a closed convex set $\Conv$ containing $0$ and additively closed, given by
a strongly quasiconvex function $\varphi$ by the formula (\ref{fookepl}), $P$ and $R=I-P$
are mutually polar retractions onto $\Conv$ and $R{\R^n}$. 
We start with a a projection $P$ . What can we say about $R$?

From the simple geometric shape of $\Conv$ similar conditions on
$R\R^n$ may be expected. But it is not the case, even if $\Conv$ is a closed convex cone and $\varphi$
is positively homogeneous, strongly quasiconvex, symmetric, smooth function.

While the structural
conditions on $\Conv$ are very simple, the answers on the simple questions
on $R$ and $R\R^n$ seem rather difficult.

\begin{remark}\label{kone}

	If $\Conv$ is a convex cone and $\varphi$ is positively homogeneous, it 
	follows
from the continuity of $\varphi$ that $P$ is continuous (Proposition 4 in
\cite{FN1}) 
and $R\R^n$ is a closed cone which
in general is not a convex one. In this note we are concerned on the simple
problem: {\bf when is this cone convex?}

\end{remark}

Following the general approach in \cite{FN} and \cite {FN1} we have to develop
a geometrical method in the absence of analytic formalism.

%%%%%%%%%%%%%%%%%%%%%%%%%%%%%%%%%%%%%%%%%%%%%%%%%%%%%%%%%%%%%%%%%%%%%%%%%%%%%%%%%%%%%%%%%%%%%%%%

\section{Projection on a halfspace}

In \cite[Proposition 10]{FN}, we provided a formula for projecting onto a hyperplane. It 
appears that applying the analytical method encounters significant 
difficulties when attempting to project onto more general cones.

The geometric perspective offers a different approach to the problem. While the projection onto a halfspace using geometric methods may hold some intrinsic interest, it is primarily included here as preparatory material for further investigation, particularly in cases where the method described in \cite{FN} proves insufficient.

The mapping $A: \R^n \to \R^n$ will be called \emph{convex additive} on the convex set $\mc M$
if for any two elements $x$ and $y$ in $\mc M$ and any $t\in [0,1]$ it holds
$$A(tx+(1-t)y)=tAx + (1-t)Ay.$$

\begin{assertion} \label{a}
If $A$ is convex additive on the convex set $\mc M$, then $A(\mc M)$ is convex.
\end{assertion}

\begin{assertion}\label{b}
Take the smooth, closed convex surface $\Si$  with $0$ in its interior. Then, the
Minkowski $\si$ functional (or gauge) with center $0$ will be an asymmetric norm (see for 
example \cite{Rockafellar1970}). That is, it
satisfies the conditions (i)-(iii) of Definition \ref{metric}. 
\end{assertion}
 
The gauge $\si$ determines a \emph{distance} of the points $x$ and $y$ by $d(x,y)=\si (x-y).$
If  $\Si$ is symmetric, then $d(x,y)=d(y,x)$ and the gauge is a norm.

Let $$\Hyp=\{x\in \R^n: \lng a,x \rng =0\}$$ with $a\in \R^n , \|a\|=1$ 
be the \emph{hyperplane through $0$ with normal $a$}. It defines two closed halfspaces:
$$\Hyp_+=\{x\in \R^n : \lng a,x\rng \geq 0\},\quad 
\Hyp_-=\{x\in \R^n : \lng a,x\rng \leq 0\},$$
and respectively two open halfspaces $$\Hyp^+=\{x\in \R^n : \lng a,x\rng >
0\},\quad \Hyp^-=\{x\in \R^n : \lng a,x\rng < 0\}.$$

Consider the closed halfspace $\Hyp_-$ and take $x\in \Hyp^+$. 
The projection by  $\si$ of $x$ to $\Hyp_-$ is geometrically realized
as follows: Translate $\Si$ with the center in $x$ : $ x+\Si$ and dilate or contract 
homothetically with center $x$ the translated sphere $x+\Si$ until 
its homothetical image touches $\Hyp$ in a single point. Denote this image by $\Si_x$.
The touching point is precisely the projection of $x$ by $\si$  on 
$\Hyp_-$, that is, $Px=\Si_x\cap \Hyp$. 
The \emph{distance of $x$ to $\Hyp_-$} is $d(x,\Hyp_-)=\si(x-Px).$

\begin{assertion}\label{c}
$$\{w\in \R^n:d(w,\Hyp_-)=d(x,\Hyp_-)\}=x+\Hyp.$$
If $x\not=y$ are two points on $x+\Hyp$ then $Px \not= Py.$
\end{assertion}
\begin{proof}
Let $x$ and $w$ be two points of distance $d$ from $H_-$.
Then, \[Px=\Hyp\cap (x+d\Si),\quad Py =\Hyp\cap (y+d\Si).\] Since $(x+d\Si) \equiv (w+d\Si)$, 
the line segment $[x,w]$ will be parallel to $\Hyp$ and $w \in x+\Hyp.$
\end{proof}

\begin{assertion}\label{konad}

The projection with respect to an asymmetric norm  $\si$ to the halfspace
$\Hyp_-$ is  convex additive on $\Hyp_+$.

\end{assertion}

\begin{proof}

Let $ x,\,y\in \Hyp^+$. 
Then, take the homotopy with center $y$ of $y+\Si$ until its
intersects $\Hyp$ in a single point which will be the projection $Py$ of $y$.
Denote
by $\Si_x$  the homothopical image of $x+\Si$, with center $x$ which realizes the projection of $x$ (that is $Px=\Si_x\cap \Hyp$).
and by $\Si_y$ the homothopical image of $y+\Si$ which realizes the projection of $y$  (that is $Py =  \Si_p\cap \Hyp$).
These two spheres $\Si_x$ and $\Si_y$ with center $x$ and respectively $y$ will be obviously similar, $\Si_x \simeq \Si_y.$ Take the
segment $[x,y]=\{tx + (1-t)y: t\in [0,1]\}$ .  A simple geometric reasoning shows that
$t\Si_x +(1-t) \Si_y = \Si_{tx+(1-t)y}$, where $\Si_{tx+(1-t)y}$ is the
sphere realizing $P(tx+(1-t)y)$. The geometric picture shows that
$P(tx+(1-t)y)=tPx+(1-t)Py.$
If $ x,\,y\in \Hyp$ the assertion is trivial. If $x\in \Hyp,\,y\in \Hyp^+$ then
$P([x,y])=[x,Py]$ and
we can apply the same geometric idea to prove the assertion.

\end{proof}

\begin{theorem}\label{felter}
If $ P$ is the projection on $\Hyp_-$ associated to a smooth  asymmetric norm $\si$,
then $P$ and $R=I-P$ are mutually polar retractions on convex cones with $R\R^n$
a semiline in $\Hyp_+$ issuing from $0$.
\end{theorem}

\begin{proof}

	Let $x\in \Hyp^+$ be the point which projects in $0$ with $d(x,H_-)=1$. Then, 
$x+\Si$ will be tangent to $\Hyp$ at the point $0$. 
Each member of the family of homothetic images with center $0$ of $x+\Si$ will de
tangent to $\Hyp$ in $0$, and hence its center will be projected in $0$. 
The center of a such member will be on the ray $[0,x $ issuing from $0$ through $x$.
Hence, each point in this halfline $[0,x = \{tx, t\in \R_+\}$ will projected in $0$.
There is no other point in $\Hyp^+$ projecting in $0$ since $\Hyp+y$ contains by
Assertion \ref{c} only a point projecting in $0$. Thus, 
$$[0,x =\{w\in \R^n: Pw = 0\}=P^{-1}(0)=R\R^n.$$
Since $[0,x $ is convex, $P$ and $R$ are mutually polar retractions on convex cones.

\end{proof}

\begin{assertion}\label{kolcs}
Each ray $[0,y$, $y\not= 0$  is the polar of a well defined halfspace $\Hyp_-$.
	Calling $[0,y$ the $\si$ray of $\Hyp_-$, this shows that there is a one to one correspondence
between $\si$rays and closed halfspaces. We will call the halfspace defined this way the
$\si$halfspace of the $\si$ray $[0,y $. We call the ray and halfspace related this way
\emph{$\si$pair}.
\end{assertion}

\begin{proof}
	Let $y$ be at the distance $1$ from $0$. Then, $0\in y+\Si$. Let $\Hyp$ be the hyperplane
tangent to $y+\Si$ at $0$ and $\Hyp_-$ the closed halfspace determined by $\Hyp$ which does not contain $y$.
\end{proof}

%%%%%%%%%%%%%%%%%%%%%%%%%%%%%%%%%%%%%%%%%%%%%%%%%%%%%%%%%%%%%%%%%%%%%%%%%%%%%%%%%%%%%%%%%%%%%%%%%%%%%%%%%%%%%%

\section{Symmetric projection on a minimal wedge}

The word "symmetric" in the title refers to the fact that we
carry our theory in this section regarding a \emph{ strictly convex, 
smooth and symmetric surface $\Si$,} that is, the defined $\si$ will be strictly
convex smooth norm.

We will use next in the proofs some obvious specific facts 
concerning  such a $\Si$ listed below.

\begin{assertion}\label{szimnor}
If $x\in \Si$
 then $\nabla_\Si (-x)= -\nabla_\Si x$
and $\nabla_{\Si+x}0 = -\nabla_\Si x$ since $\Si+x$ is the translation of $\Si$
by which $0$ corresponds to $-x$. (Assertion \ref{utolag}). 
\end{assertion}

A \emph{meridian of $\Si$} is the set of form  $\Si_\mc D = \Si \cap \mc D$ with $\mc D\subseteq \R^n$
a 2 dimensional subspace of $\R^n$. A connected part of a meridian is
called \emph{meridian arc}.

If $a$ is the normal of the 2 dimensional subspace $\mc D$ of $\R^n$, then
$\Si_D^+ = \{x\in \Si : \lng a,x\rng > 0$\} is the open north,
$\Si_D^- = \{x\in \Si : \lng a,x\rng < 0$\} the open south 
halfsphere (with respect to $\mc D$) of $\Si$.

\begin{assertion} \label{szembe}
If $ x,\,y\in \Si$ with $\nabla_\Sigma y= - \nabla_\Sigma x,$ then $y=-x$.
\end{assertion}

The set of the form

$$ \omega = \Hyp^1_-\cap \Hyp^2_-\subseteq \R^n$$
with $\Hyp^1$ and $\Hyp^2$ different hyperplanes in $\R^n$ will be caolled a \emph{minimal wedge}.
The $n-2$-dimensional subspace $\Hyp^1 \cap \Hyp^2$ will be said the \emph{edge of $\omega$}.

Denote by $\Omega$ the family of all the minimal wedges in $\R^n$.

\begin{lemma}\label{yy}
Suppose that $\si$ is a strictly convex smooth asymmetric norm.
Let $\mc C'$ and $\mc C"'$ be two convex cones with $\mc C''\subseteq \mc C'$.
Denote the projections with respect to $\si$ into $\mc C'$ and $\mc C''$ by $P'$ and $P''$, 
respectively. Then, $\ker(P')\subseteq \ker(P'').$
\end{lemma}
\begin{proof}
	Let $x \in \ker(P'),$ that is $$\si(x) =\min\{\si(x-w): w\in \mc \Conv'\}.$$
On the other hand 
$$\si (x) = \si (x-0)\geq \min\{\si(x-v): v\in \mc C''\} \geq \min\{\si(x-w: w\in \mc \Conv'\}=\si(x) 
,$$
since $0\in \mc C''$ and $\mc C''\subseteq \mc \Conv'$. Therefore, $x\in
\ker(P'').$
\end{proof}

\begin{proposition}\label{w1}
If $\omega = \Hyp^1_-\cap \Hyp^2_-\subseteq \R^n$ is a minimal wedge with
$\Xi =\{\Hyp^\lambda_-\}$  the family of all the halfspaces containing $\omega$, then
denoting $P = P_\omega$, $R=R_\omega:=I-P$, the polar set $R\R^n$ of $\omega$ can  be 
constructed as follows:
For each $\Hyp^\lambda$ consider its $\si$ray $[0, x_\lambda: \si(x_\lambda)=1.$ Moving $\Hyp^\lambda$
continuously from $\lambda =1$ to $\lambda= 2$, $x_\lambda$ will describe a 
simple arc $\gamma \subseteq \Si$ and
$$R\R^n:=\cup \{[0,x_\lambda : x_\lambda \in \gamma\}.$$
\end{proposition}

\begin{proof}
Denote

$$\mc A:=\cup \{[0,x_\lambda : x_\lambda \in \gamma\}.$$

	For every $\Hyp^\lambda_-$, $[0, x_\lambda$ is a set projected  on this halfspace in $0$.
Since $\omega \subseteq \Hyp^\lambda_-$ and $0\in \omega\cap \Hyp^\lambda_-$ we
can apply Lemma \ref{yy} and the second formula of \eqref{foo1} to conclude that 
$R_{\Hyp^\lambda}(\R^n) \subseteq R_\omega (\R^n)=R\R^n.$
On the other hand, by using Theorem \ref{felter}, we have 
$$R_{\Hyp^\lambda}(\R^n) = [0, x_\lambda.$$
	Hence, $\mc A\subseteq R\R^n.$

If $z\in R\R^n$ is a point projecting in $0\in \omega$, then if $z+ \Si$ 
is the sphere realizing the projection, the tangent
hyperplane $\Hyp$ to this sphere in $0$ will be a separating hyperplanne
of this sphere and $\omega$ hence $H_-$ suport $\omega$ at $0$ and must 
be a member of the family $\Xi$. Thus, $R\R^n\subseteq \mc A.$

\end{proof}

\begin{remark}\label{yyy}
Using more consistently  the terminology introduced in Assertion \ref{kolcs},
we can formulate this theorem by saying that \emph{the polar  of $\omega$ is
the union of the $\si$rays of the supporting halfspaces of $\omega$}. 
\end{remark}

\begin{proposition}\label{w2}

Suppose that $\gamma$ is a simple  arc containing its endpoints in an open
halfsphere of $\Si$. Suppose that $\nabla(\gamma)=\{\nabla x: x\in \gamma\}$
is in a two dimensional subspace $\mc D$ of $\R^n$. Then,
$$\omega =\bigcap\lf\{\Hyp_-(-\nabla x):\,x\in \gamma \rg\}$$ 
is a minimal wedge with the edge $\mc{D}^\perp$. 
Furthermore, denoting $P=P_\omega$ and $R=I-P$, we have
$$ R\R^n= \cup \lf\{[0,x : x\in \gamma \rg\}.$$

	In other words, $R\R^n$ is the cone generated by $\gamma$. (See also Remark \ref{kone}.)

\end{proposition}

\begin{proof}

	Since $\gamma$ is compact and $\nabla$ is continuous on $\Si$, the later being smooth,  
$\nabla (\gamma)$ is a closed set, which, by the hypothesis and Assertion
\ref{szembe}, cannot contain diametrically opposite
vectors. The set $\omega$ is a minimal wedge with the edge $\mc{D}^\perp$.
Let $x\in \gamma$. Let $\Hyp(\nabla (x))+x$ be  the hyperplane tangent 
to $\Si$ in $x$ and with the normal $\nabla (x)$. 
Then, the sphere $\Si+x$ will be tangent to the hyperplane $\Hyp(\nabla (x))$ at $0$, and by
Assertion \ref{szimnor} the normal of $\Si + x$ at $0$ will be $-\nabla (x)$. Thus
$[0,x $ will be the $\si$ray of $\Hyp_-(-\nabla (x))$. Since 
	$\Hyp_-(-\nabla(x))$ supports  $\omega$, by using Remark \ref{yyy}, $[0,x \subseteq R\R^n$.
A reasoning similar to the one in the proof of
Proposition \ref{w1} concludes the proof.

\end{proof}

Let $\Sigma\subseteq\R^n$ be a bounded closed strictly convex (symmetric with
respect to 0) surface. The arc $\gamma$ on a meridian of $\Si$ is said
\emph{coherent} if $\dim(\spa(\nabla(\gamma))=2.$

\begin{theorem}\label{ek}
Let $\omega \in \Omega$, $P_\omega$ be the $\si$-projection into $\omega$ and  
$R_\omega = I- P_\omega.$
Then, the following assertions are equivalent : 
\begin{enumerate}
\item $P_\omega$ and $R_\omega$ are mutually polar retractions on convex cones;
\item $P_\omega$ is convex additive on $\R_\omega (\R^n)$;
\item Let $\omega =\Hyp^1_-\cap \Hyp^2_-$, 
	$$u\in \si\textrm{\em ray}(\Hyp^1_-),\textrm{ }d(0,u)=1,$$
	$$v\in \si\textrm{\em ray}(\Hyp^2_-),\textrm{ }d(0,v)=1,$$
	(hence $u,\,v \in \Si$). Then,
the meridian arc $\gamma$ on $\Si$ with the endpoints $u$ and $v$ is coherent.
\end{enumerate}
\end{theorem}

\begin{proof}
$\,$

\begin{enumerate}
	\item[ ] 1$\iff$2: 
From the second formula in \eqref{foo1}, we know that $R_\omega(\R^n) =
\ker(P_\omega)$. 
Hence, the convexity of the latter
set is equivalent to \[u,\,v\in \ker(P_\omega)\iff tu+(1-t)v \in \ker(P_\omega)\quad \forall 
t\in [0,1].\] This holds if and only if $P_\omega$ is convex additive on
$R_\omega(\R^n) = \ker(P_\omega)$, since 
\[P_\omega (tu+(1-t)v)=t P_\omega u+(1-t) P_\omega v = 0.\]
\item[ ] 1$\iff$3:
According to Propositions \ref{w1} and \ref{w2}, $R\R^n$ will be
a convex cone if and only if $\dim(\spa(\nabla(\gamma))) = 2.$ 
\end{enumerate}

\end{proof}

\begin{lemma}\label{fontos}

If $\Si\subseteq \R^n$ is a  bounded closed strictly convex (symmetric with respect to $0$)  surface
and $A:\R^n\to \R^n$ is a nonsingular linear map, then
\begin{enumerate}
\item $A\Si$ is a  bounded closed strictly convex (symmetric with respect to $0$)  surface;
\item If $\Si_\mc D= \Si\cap \mc D$ is a meridian of $\Si$, then 
$A\Si_{A\mc D}=A\Si\cap A\mc D$ is a meridian of $A\Si$;
\item Coherent meridian arcs on $\Si$ are transformed by $A$ in coherent meridian arcs on $A\Si$.
\end{enumerate}

\end{lemma}
\begin{proof}

1. $A\Si$ is obviously convex. If $\Hyp$ is the supporting hyperplane of $\Si$ in $x$, that is
$\Si\cap \Hyp= x$ then $A\Si \cap A\Hyp = Ax$ and $A\Hyp$ is hyperplane.

2. Is obvious.

3. Let $\gamma $ be a coherent arc on $\Si$. If $\Hyp^1,\,\Hyp^2,\,\Hyp^3$ are tangent 
hyperplanes to $\Si$ at different points $x^1,\,x^2,\,x^3 \in \gamma$, then by definition
$\dim\lf((\Hyp^1-x^1) \cap (\Hyp^2-x^2) \cap (\Hyp^3-x^3)\rg) = n-2.$ Hence,

\begin{align*} & \dim (A\Hyp^1-Ax^1) \cap (A\Hyp^2-Ax^2) \cap (A\Hyp^3-Ax^3)\\  
= & \dim A(\Hyp^1-x^1) \cap A(\Hyp^2-x^2) \cap A(\Hyp^3-x^3)\\
= & \dim A((\Hyp^1-x^1) \cap (\Hyp^2-x^2) \cap (\Hyp^3-x^3))\\  
= & \dim(\Hyp^1-x^1) \cap (\Hyp^2-x^2) \cap (\Hyp^3-x^3) = n-2.\end{align*}

Thus, the gradients to $A\Si$  at the points $Ax^1,\,Ax^2,\,Ax^3\in A(\gamma)$
must be in the subspace orthogonal to an $(n-2)$-dimensional 
subspace of $\R^n$, hence they must be coplanar. Hence, $\nabla (A(\gamma))$
is coplanar.

\end{proof}

%%%%%%%%%%%%%%%%%%%%%%%%%%%%%%%%%%%%%%%%%%%%%%%%%%%%%%%%%%%%%%%%%%%%%%%%%%%%%%%%%%%%%%%%%%%%%%%%%%%%%%%%%%

\section{Global bipolar projector}

The assertion in Remark \ref{yyy} regarding minimal wedges has its
generalization for  general closed convex cones too:

\begin{proposition}\label{altalanos}
Suppose that $\si$ is a smooth, strictly convex norm, $\Si$ is its unit sphere, 
with $P = P^\si$ the associated projection, and $R=I-P.$
Let $\mc  C$ be a closed convex cone. Then, $R\R^n$ is precisely the
union of the $\si$rays with $\si$halfspaces supporting $\mc C$ at $0$.
\end{proposition}

\begin{proof}
Let $x\in R\R^n,\,d(x,0)=1.$ Then, the sphere $\Si+x$ is tangent to $\mc C$ at $0$.
If $\Hyp$ is the tangent hyperplane of $\Si+x$ at $0$, then $\Hyp_-$ will be the
$\si$halfspace of the $\si$ray $[0,x $ and will support $\mc C$ at $0$.
\end{proof}

 The mapping $P^\si$ is called \emph{global bipolar projector} and the corresponding norm $\si$
a \emph{a global bipolar norm} if $P_\mc C$ and $R_\mc C=I-P_\mc C$ are
mutually polar retractions on convex cones for every closed convex cone $\mc C$,
that is if $\ker(P_\mc C) = R_\mc C(R^n)$ is a convex cone for 
any closed convex cone $\mc C$.

\begin{proposition}\label{7peld}
	Let $\mc P\subseteq \R^m$ a pointed closed convex cone, $A$ a symmetric 
	positive definite matrix and the norm $\|\cdot\|_A$ on $\R^m$ defined by 
	$\|x\|_A=\sqrt{\lng Ax,x\rng}$. Denote $P = P^{\|\cdot\|_A}_{\mc P}$ and 
	$R=I-P$. Then, $P,R$ is a pair of mutually polar retractions on convex cones.
\end{proposition}

\begin{proof}
	Clearly, $P\R^n=\mc P$ and $P^2=P$. Since $P+R=I$ and since according to
	Theorem \ref{footet} $PR=0$, 
	it follows that $R^2=R$. Since $\mc
	R = \ker(P)$, it is 
	enough to show that $\ker(P)$ is a pointed convex cone. Denote $\varphi=\|\cdot\|_A$. We 
	know (from Theorem 5 \cite{FN}) that 
	\[\ker(P)=\lf\{x\in\R^m\setminus\{0\}:\nabla\varphi(x)\in\mc
	P\cc\rg\}\cup\{0\}.\] On the other hand, for any non-zero $x\in\R^m$ we have
	$\nabla\varphi(x)\in\mc P\cc$ if and only if 
	\[Ax\in\mc P^\circ.\] Hence, $\mc R=\ker(P)=A^{-1}(\mc P^\circ)$ is a pointed
	convex cone. 
\end{proof}

\begin{theorem}\label{univers}

Consider the smooth, symmetric, strictly convex norm $\si$. 
The following assertions are equivalent:
\begin{enumerate}[(i)]
	\item  The mapping $P:=P^\si$ is a universal bipolar projector, that is, 
$P_\mc C$ and $R_\mc C=I- P_\mc C$ are mutually polar 
retractions on convex cones for every closed convex cone $\mc C$;
\item $\ker (P_\mc C)$ is convex for every closed convex cone $\mc C$;
\item $P_\mc C$ is convex additive on $R_\mc C(\R^n)$ for every
closed convex cone $\mc C$;
\item [(iv)] All the meridians of the unit sphere $\Si=\{x\in \R^n: \si(x)=1\}$
of the norm $\si$ are coherent;
\item [(v)] $P_\omega$ and $R_\omega = I-P_\omega$ are mutually polar retractions
on convex convex cones for each minimal wedge $\omega \in \Omega$.
\end{enumerate}
\end{theorem}

\begin{proof} Bearing in mind \eqref{foo1}$_2$, the equivalences (i)$\iff$(ii)$\iff$(iii) are 
	immediate consequences of the definitions and the basic properties of mutual 
	retractions.

The equivalence (iv)$\iff$(v) is the consequence of Theorem \ref{ek}.

Let us verify the equivalence (ii)$\iff$(v).

(ii)$\implies$(v) is  obvious.

Let us show that (v)$\implies$(ii).

 Suppose that 
 $u,v\in R_\mc C(\R^n) \setminus \{0\}$, where $u$ and $v$ are linearly independent.
 Suppose that $[0, u$ andf $[0,v $ are $\si$rays with $\Hyp_-(u)$ and $\Hyp_-(v)$ the corresponding $\si$halfspaces. Then, by Proposition \ref{altalanos}, $\omega = \Hyp^u_-\cap \Hyp^v_-$ is a minimal wedge and
$\mc C\subseteq \omega$. Use Lemma \ref{yy} to conclude that
$\ker(P_\omega)\subseteq \ker(P_\mc C).$ 

Since $\ker(P_\omega)$ is convex by assumption, it follows that $\ker(P_\mc C)$ is convex too.

\end{proof}

\begin{corollary} The following statements hold.

\begin{enumerate}
\item [(i)]
If $\si$ is a global bipolar norm, then so is $A\si$ for any nonsingular linear operator
$A:\R^n\to\R^n$. 
\item [(ii)] If  $\si(x)= x^TAx$ with $A$ a symmetric positive definite matrix,
then all the meridians of $\Si= \{x\in \R^n:x^TAx=1\}$ are coherent and we
get by Theorem \ref{univers} another proof of Proposition \ref{7peld}.
\end{enumerate}
\end{corollary}

\begin{proof}
The assertion (i) is via Theorem \ref{univers} the consequence of Lemma \ref{fontos}.

(ii) If $\si(x)= x^TAx$ then $A^{-1/2} \si$ (where $A^{-1/2}$ is the inverse of the square root of $A$)
becomes the Euclidean metric. Hence, the meridians of the unit sphere of $A^{-1/2}\si$ are all coherent. By
Lemma \ref{fontos} all the meridians of $\Si$ are coherent and hence $\si$ is a
global bipolar norm.
\end{proof}

\end{document}